\documentclass{my-jca32x}

\usepackage{mathrsfs}
\usepackage{amssymb}
\usepackage{url}

\begin{document}
	
\setcounter{page}{1}


\title{Variational Formulations for Fractional Integral and Differential Equations}

\author{
\ftauthor{Delfim F. M. Torres%
\footnote{The research was funded by a grant from the Portuguese 
Foundation for Science and Technology (FCT) through the R\&D Unit 
CIDMA (Project No. UID/04106/2025).}}
\ftaddress{Center for Research and Development in Mathematics and Applications (CIDMA),\\
Department of Mathematics, University of Aveiro, 3810-193, Aveiro, Portugal\\[-2pt]
delfim@ua.pt}}

\def\runtit{D.F.M.\,Torres / Variational Formulations ...}


\maketitle


\ReceivedAccepted{April 22, 2025}{June 25, 2026}


\abstract{

{\bf (Dedicated to Alexander Plakhov
on the occasion of his 65th anniversary)}
	
Given a fractional-order linear equation $L^\alpha u = f$,
we define an appropriate symmetric bilinear form so that
the fractional operator $L^\alpha$ is symmetric with respect
to that bilinear form. Using the bilinear form, we then define 
a functional of the fractional calculus of variations 
proving that the solutions of the given fractional-order equation are 
critical points of the fractional variational functional.
In the case of fractional integral equations, the provided
bilinear form is non-degenerate, and all critical points
are solutions of the given equation.
In the case of fractional differential equations, 
a relation with the least-squares method is obtained.}


\vskip-2mm
\keyw{Fractional-order integral equations, 
fractional--order differential equations,
fractional calculus of variations,
fractional Euler--Lagrange equations,
inverse problems.}


\AMSsc{2020}{34A08, 49K05, 49N45.}


\section{Introduction}

In recent years, fractional (non-integer order)
differential and integral equations have attracted much attention, 
see, e.g., \cite{MR4163834,MR4648926} and references therein.

The fractional Calculus of Variations 
was introduced in 1996--1997, with the work of Riewe in Physics, 
to obtain a Lagrangian for a dissipative system with a damping 
force proportional to the velocity \cite{CD:Riewe:1996,CD:Riewe:1997}.
Recently, Torres introduced a fractional-order 
calculus of variations by providing new formulas of integration by parts
that involve only left or only right fractional operators \cite{MR4740251}. 
This allowed him to obtain, in contrast with all previously available results
of the literature, Euler--Lagrange equations that involve only left 
or only right fractional operators, thus providing a long-sought answer 
to some non-causality concerns on the fractional calculus of variations 
\cite{MR3390552}. Moreover, and as a consequence of the new perspective 
to the subject, the first example of a Lagrangian involving only left 
fractional derivatives was obtained for which the respective Euler--Lagrange 
equation coincide with the equation of motion of a dissipative system \cite{MR4740251}.

Here, we interpret the fractional integration by parts formulas of \cite{MR4740251}
as the symmetry of the fractional-order operator with respect to suitable bilinear forms.
This new interpretation allows us to associate fractional integral 
and differential equations to fractional-order
problems of the calculus of variations, in line with the classical (integer-order)
calculus of variations \cite{MR1502361,MR1501455,MR0004740,MR0350177,MR1502352}.

To the best of our knowledge, the interpretation of the
one-sided fractional integration-by-parts formulas of
\cite{MR4740251}, as symmetry relations associated with
bilinear forms, together with the resulting variational
formulations, has not been previously reported.
In fact, while the (direct) 
fractional calculus of variations is well established,
see, e.g., the books \cite{book:frac:ICP2,MyID:404,book:adv:FCV,book:frac},
the inverse problem is less studied. Precisely, the inverse problem
of the fractional calculus of variations has only been solved
for specific fractional partial differential equations 
-- the fractional wave equation, the fractional diffusion equation, 
and the incompressible Stokes' equations \cite{MR2557007};
and for fractional ordinary differential equations involving simultaneously
left and right fractional operators \cite{MR2824707,MR3091899,MR3715699}.
Here, using the recent approach of \cite{MR4740251}, we are able to deal
with more standard fractional integral and fractional ordinary differential 
equations involving only left (or only right) operators.

The paper is organized as follows. In Section~\ref{sec:02}, we 
solve the inverse problem of the fractional calculus of variations
for fractional integral equations $\mathcal{I}^{\alpha} u = f$
(see Theorem~\ref{thm:01}), obtaining also new insights 
to the well-studied direct problem of the fractional
calculus of variations (see Corollary~\ref{cor:01}). 
A fractional variational formulation for fractional 
differential equations $\mathcal{D}^{\alpha} u = f$
is then obtained in Section~\ref{sec:03} (see Theorem~\ref{thm:02}).
It should be observed that the situations 
studied in Sections~\ref{sec:02} and \ref{sec:03}
are intrinsically different and require the use of different
bilinear forms so that the considered fractional operators 
become symmetric with respect to the bilinear form. 
Moreover, our variational formulation for fractional integral operators 
$\mathcal{I}^{\alpha}$ is stronger than the one obtained for
fractional differential operators $\mathcal{D}^{\alpha}$,
in the sense Theorem~\ref{thm:01} provides a necessary and sufficient
condition while, in contrast, Theorem~\ref{thm:02} is only an implication.
We end our paper with Section~\ref{sec:rem:02},
establishing a link with the least-squares method.

Along the text, $\alpha$'s are always real numbers 
between zero and one: $\alpha \in (0,1)$.
Our results are formal and the reader 
is assumed to be familiar with 
the classical fractional calculus
of Riemann--Liouville \cite{MR1347689}.
In particular, the reader
is referred to \cite{MR1347689}
for the domains and ranges of the considered operators.


\section{Variational Formulation of Fractional Integral Equations}
\label{sec:02}

In this section, we use the fractional calculus of variations of \cite{MR4740251}
to obtain a variational formulation for a given fractional integral equation
or a given system of fractional integral equations of the form
\begin{equation}
\label{eq:01}
\mathcal{I}^{\alpha} u = f,
\end{equation}
where $u$ denotes a scalar or vector integrable function on the interval $[a, b]$
and $\mathcal{I}^{\alpha}$ denotes the left $\mathcal{I}^{\alpha}_{a+}$ 
or the right $\mathcal{I}^{\alpha}_{b-}$ 
fractional integral of order $\alpha$.
Our aim here is to obtain a functional $\mathcal{J}_{I}^{\alpha}[u]$ 
such that its critical points coincide with the solutions of the given fractional
integral equation \eqref{eq:01}. In other words, while in \cite{MR4740251}
Torres has introduced a fractional calculus of variations and studied its direct
problem, here we are interested in its inverse problem: to begin with
a fractional equation and obtain the associated fractional variational functional.

\subsection{The bilinear form associated with the fractional integral equation}

Let us define the functional $\left<\cdot, \cdot\right>$ by
\begin{equation}
\label{eq:01b}
\left<v,u\right> := \int_{a}^{b} v(t) u^*(t) dt,
\end{equation}
where $u^*$ is the dual function of $u$ ($u^*(t) = u(b-t+a)$,
see Definition 3.1 of \cite{MR4740251}). 

We use here the symbol $\left<\cdot, \cdot\right>$ to stress that the functional \eqref{eq:01b}
satisfies the same properties of the scalar product of two elements of an Hilbert space.
The proof of Proposition~\ref{prop:01} is a simple exercise and is left to the reader.

\begin{prop}
\label{prop:01}	
The functional $\left<v,u\right>$ is a bilinear form (linear in both arguments $v$ and $u$), and
symmetric ($\left<v,u\right> = \left<u,v\right>$). 
Moreover, $\left<v,u\right>$ is non-degenerate 
(if $\left<v,u\right> = 0$ for all $v$, then $u=0$; 
if $\left<v,u\right> = 0$ for all $u$, then $v=0$).
In the space of functions satisfying the symmetry condition $u(t)=u^*(t)$,
the functional $\left<v,u\right>$ is also positive definite ($\left<u,u\right> \geq 0$ 
for all $u$, with $\left<u,u\right> = 0$ if, and only if, $u=0$). 
\end{prop}

\subsection{Integration by parts for fractional integrals}

We now provide a new look to the formulas 
of integration by parts for fractional integrals
proved in \cite{MR4740251}.

\begin{prop}
\label{prop:02}
Let $\mathcal{I}^{\alpha}$ be a given linear fractional integral operator.
Then $\mathcal{I}^{\alpha}$ is symmetric with respect to the bilinear
form \eqref{eq:01b}, that is,
\begin{equation}
\label{eq:02}
\left<\mathcal{I}^{\alpha} u_1,u_2\right> 
= \left<\mathcal{I}^{\alpha} u_2,u_1\right>.
\end{equation}
\end{prop}

\begin{proof}
Relation \eqref{eq:02} follows directly from
our definition \eqref{eq:01b} and
Theorems~3.12 and 3.13 of \cite{MR4740251}.
\end{proof}

The importance of Proposition~\ref{prop:02} is that it allows one to reinterpret
fractional integration-by-parts formulas as symmetry properties of the
fractional operator. Indeed, in the classical theory, the identity
$$
\int_a^b (Lu)v dt
=
\int_a^b u(Lv) dt
$$
is often viewed as a manifestation of self-adjointness.
Similarly, relation \eqref{eq:02} shows that the fractional integral 
operator $\mathcal{I}^\alpha$ becomes symmetric when viewed through the bilinear form
$\langle\cdot,\cdot\rangle$.

This observation provides a bridge between fractional calculus and the
classical inverse problem of the calculus of variations.
Rather than constructing a Lagrangian by ad hoc methods, one starts from
a symmetry property induced by the fractional integration-by-parts
formula and obtains the variational functional in a systematic way.

\subsection{Variational formulation for the fractional integral equation}

Our first main result follows 
from Propositions~\ref{prop:01} and \ref{prop:02}.
The bilinear forms introduced in this work play a role analogous to the
inner products that appear in the classical variational formulation of
differential equations. In the standard Hilbert-space setting, a linear
operator can often be represented, variationally, whenever it is
symmetric with respect to a suitable bilinear form.
This observation lies at the heart of many variational methods in
mathematical physics and numerical analysis \cite{MR1115205,MR0350177}.

For fractional operators involving only left or only right
Riemann--Liouville operators, the standard $L^2$ inner product does not
provide the appropriate framework. The duality theory developed in
\cite{MR4740251} suggests replacing the classical scalar product by
bilinear forms involving the dual transformation
$$
u \mapsto u^*,
\qquad
u^*(t)=u(b-t+a).
$$
The resulting bilinear forms naturally encode the duality between left
and right fractional operators and transform the integration-by-parts
formulas of fractional calculus into symmetry relations.

From this perspective, the bilinear forms are not auxiliary objects but
rather the geometric structures that make a variational interpretation
possible.

In some sense, our next result can be interpreted as a fractional analogue
of the classical principle that symmetric operators admit natural
quadratic variational formulations
\cite{MR1502361,MR0004740}.
The strategy adopted here follows a classical principle of
variational analysis: rather than searching directly for a
Lagrangian, one first identifies a bilinear form with respect to
which the governing operator is symmetric. Once such a structure
is available, the associated quadratic functional arises naturally.
This philosophy goes back to the classical theory of operator
equations and weak formulations: see, for instance,
\cite{MR0350177}.

\begin{theo}
\label{thm:01}
A function $u$ is a solution of the fractional integral equation \eqref{eq:01}
if, and only if, $u$ is a critical point of the fractional variational functional
\begin{equation}
\label{eq:03}
\mathcal{J}_{I}^{\alpha}[u] :=
\frac{1}{2}\left<\mathcal{I}^{\alpha} u,u\right> - \left<f,u\right>.
\end{equation}
\end{theo}

\begin{proof}
We perform the first variation of $\mathcal{J}_{I}^{\alpha}[u]$.
Let $\Phi(\varepsilon) = \mathcal{J}_{I}^{\alpha}[u + \varepsilon v]$.
By definition, function $u$ is a critical point of $\mathcal{J}_{I}^{\alpha}[u]$
if, and only if, $\varepsilon = 0$ is a critical point of the real function
$\Phi(\varepsilon)$, that is, $\Phi'(0) = 0$. By \eqref{eq:03} and
the bilinearity of $\left<\cdot, \cdot \right>$, we have
$$
\Phi(\varepsilon) = \frac{1}{2} \left<\mathcal{I}^{\alpha} u, u\right>
+ \frac{1}{2} \varepsilon \left<\mathcal{I}^{\alpha} u, v\right>
+ \frac{1}{2} \varepsilon \left<\mathcal{I}^{\alpha} v, u\right>
+ \frac{1}{2} \varepsilon^2 \left<\mathcal{I}^{\alpha} v, v\right>
- \left<f, u\right> - \varepsilon \left<f, v\right>,
$$
and, by Proposition~\ref{prop:02}, 
$$
\Phi'(0) = 0
\Leftrightarrow
\frac{1}{2} \left<\mathcal{I}^{\alpha} u, v\right>
+ \frac{1}{2} \left<\mathcal{I}^{\alpha} v, u\right>
- \left<f, v\right> = 0
$$ 
is equivalent to
\begin{equation}
\label{eq:04}
\left<\mathcal{I}^{\alpha} u - f, v\right> = 0.
\end{equation}
Since $v$ is arbitrary, it follows from \eqref{eq:04} 
and the non-degeneracy of the bilinear form 
$\left<\cdot, \cdot\right>$ 
(see Proposition~\ref{prop:01}) that \eqref{eq:04}
is equivalent to $\mathcal{I}^{\alpha} u - f = 0$.
The proof is complete.
\end{proof}

Theorem~\ref{thm:01} not only solves the inverse fractional variational problem 
for \eqref{eq:01}, but also gives new insights to the direct problem of the fractional
calculus of variations \cite{book:frac:ICP2,MyID:404,book:adv:FCV,book:frac,MR4740251}. 
Indeed, we are not aware of any analogous result in the
existing literature.

\begin{cor}
\label{cor:01}	
Let $\mathcal{I}^{\alpha}$ be the left $\mathcal{I}^{\alpha}_{a+}$ 
or the right $\mathcal{I}^{\alpha}_{b-}$ 
Riemann--Liouville fractional integral of order $\alpha \in (0,1)$.
If $u$ is a minimizer of the fractional variational functional
\begin{equation}
\label{eq:05}
\mathcal{J}_{I}^{\alpha}[u] =
\int_{a}^{b} \left(\frac{1}{2} \left(\mathcal{I}^{\alpha} u\right)(t)-f(t)\right) u(b-t+a) \, dt,
\end{equation}
then $u$ satisfies the Euler--Lagrange equation
$\left(\mathcal{I}^{\alpha} u\right)(t) = f(t)$ 
for $t \in [a,b]$.
\end{cor}

We illustrate Theorem~\ref{thm:01} with a simple fractional integral
equation.

\begin{example}
Let $a=0$, $b=1$, and let $\mathcal{I}^\alpha=\mathcal{I}_{0+}^\alpha$ be the
left Riemann--Liouville fractional integral of order $\alpha\in(0,1)$.
Consider
\begin{equation}
\label{eq:example-integral}
\mathcal{I}_{0+}^\alpha u(t)=1,\qquad t\in[0,1].
\end{equation}
It is well known that
$$
\mathcal{I}_{0+}^{\alpha} t^{\beta-1}
=
\frac{\Gamma(\beta)}{\Gamma(\alpha+\beta)}
t^{\alpha+\beta-1},
\qquad \beta>0.
$$
Taking $\beta=1-\alpha$, we obtain
$$
\mathcal{I}_{0+}^{\alpha}
\left(
\frac{t^{-\alpha}}{\Gamma(1-\alpha)}
\right)
=1.
$$
Hence,
\begin{equation}
\label{eq:example-solution}
u(t)=\frac{t^{-\alpha}}{\Gamma(1-\alpha)}
\end{equation}
is a solution of \eqref{eq:example-integral}.
Note that
since $0<\alpha<1$, the function \eqref{eq:example-solution}
belongs to $L^1(0,1)$.
According to Theorem~\ref{thm:01}, equation
\eqref{eq:example-integral} is associated with the variational functional
\begin{equation}
\label{eq:example-functional}
\mathcal{J}_I^\alpha[u]
=
\frac{1}{2}\langle \mathcal{I}_{0+}^\alpha u,u\rangle
- \langle 1,u\rangle ,
\end{equation}
where
$$
\langle v,u\rangle
=
\int_0^1 v(t)u^*(t),dt,
\qquad
u^*(t)=u(1-t).
$$
Thus,
\begin{equation}
\label{eq:example-functional-expanded}
\mathcal{J}_I^\alpha[u]
=
\frac{1}{2}\int_0^1
\bigl(\mathcal{I}_{0+}^\alpha u\bigr)(t)u(1-t) \, dt
-
\int_0^1 u(1-t) \, dt.
\end{equation}
Let $v$ be an admissible variation and set
$$
\Phi(\varepsilon)=\mathcal{J}_I^\alpha[u+\varepsilon v].
$$
Using the bilinearity of $\langle\cdot,\cdot\rangle$ and the symmetry
property
$$
\langle \mathcal{I}_{0+}^\alpha u,v\rangle
=
\langle \mathcal{I}_{0+}^\alpha v,u\rangle ,
$$
we obtain
$$
\Phi'(0)
=
\langle \mathcal{I}_{0+}^\alpha u-1,v\rangle .
$$
Therefore, $u$ is a critical point of $J_I^\alpha$ if and only if
$$
\langle \mathcal{I}_{0+}^\alpha u-1,v\rangle=0
$$
for all admissible variations $v$. Since the bilinear form
$\langle\cdot,\cdot\rangle$ is non-degenerate, this is equivalent to
$$
\mathcal{I}_{0+}^\alpha u(t)=1,
$$
that is, to the original fractional integral equation
\eqref{eq:example-integral}.
Consequently, the function \eqref{eq:example-solution} is precisely a
critical point of the functional \eqref{eq:example-functional}. This
example shows explicitly how the proposed bilinear form transforms a
fractional integral equation into a variational problem whose critical
points coincide with the solutions of the original equation.
\end{example}


\section{Variational Formulation of Fractional Differential Equations}
\label{sec:03}

We now address the problem of obtaining a variational formulation for
a fractional differential equation
\begin{equation}
\label{eq:06}
\mathcal{D}^{\alpha} u = f,
\end{equation}
where $u$ denotes a scalar or vector integrable function on the interval $[a, b]$
and $\mathcal{D}^{\alpha}$ denotes the left $\mathcal{D}^{\alpha}_{a+}$ 
or the right $\mathcal{D}^{\alpha}_{b-}$ 
Riemann--Liouville fractional derivative of order $\alpha$. 

The situations considered in Section~\ref{sec:02} 
and now here in Section ~\ref{sec:03} 
are fundamentally different
since $\mathcal{D}^{\alpha}$ is not symmetric with respect to
the bilinear form \eqref{eq:01b}. Indeed, it follows from the
integration by parts formulas involving only left or only right 
fractional derivatives (Theorems~3.14 and 3.15 of \cite{MR4740251}) that
$\left<\mathcal{D}^{\alpha} u_1,u_2\right> 
= - \left<\mathcal{D}^{\alpha} u_2,u_1\right>$,
that is, $\mathcal{D}^{\alpha}$ is anti-symmetric (or skew-symmetric)  
with respect to the bilinear form \eqref{eq:01b}.

To solve this difficulty, and find a variational formulation for the fractional
differential equation \eqref{eq:06}, we change the bilinear form \eqref{eq:01b}
so that the fractional derivative $\mathcal{D}^{\alpha}$ becomes symmetric
with respect to the new bilinear form. For that, we propose the bilinear
form $\left(\cdot, \cdot \right)$ defined by
\begin{equation}
\label{eq:07}
\left(v,u\right) := \left<v, \mathcal{D}^{\alpha} u\right>
= \int_{a}^{b} v(t) \left(\mathcal{D}^{\alpha} u\right)^*(t) \, dt.
\end{equation}

\begin{prop}
\label{prop:03}	
The fractional derivative operator $\mathcal{D}^{\alpha}$ is symmetric
with respect to the bilinear form \eqref{eq:07}, that is, 
\begin{equation}
\label{eq:08}
\left(\mathcal{D}^{\alpha} u_1,u_2\right)
= \left(\mathcal{D}^{\alpha} u_2,u_1\right).
\end{equation}
\end{prop}

\begin{proof}
From definition \eqref{eq:07}, and because 
$\left<\cdot, \cdot\right>$ is symmetric, we have
$$
\left(\mathcal{D}^{\alpha} u_1,u_2\right)
= \left<\mathcal{D}^{\alpha} u_1, \mathcal{D}^{\alpha} u_2\right>
= \left<\mathcal{D}^{\alpha} u_2, \mathcal{D}^{\alpha} u_1\right>
= \left(\mathcal{D}^{\alpha} u_2,u_1\right),
$$
which proves the result.
\end{proof}

\begin{prop}
The bilinear form $(\cdot,\cdot)$ is degenerate.
Indeed, every function $u$ satisfying $\mathcal{D}^\alpha u=0$
belongs to its kernel.
\end{prop}

\begin{proof}
If $\mathcal{D}^\alpha u=0$, then
$$
(v,u)
=
\langle v,\mathcal{D}^\alpha u\rangle
=
0
$$
for every $v$. Hence, $u$ belongs to the kernel of $(\cdot,\cdot)$.
\end{proof}

Using the bilinear form $\left(\cdot, \cdot \right)$,
we are now in conditions to give a variational formulation 
for the fractional differential equation \eqref{eq:06}
(our next result, Theorem~\ref{thm:02}). It is important to note, 
however, that while Theorem~\ref{thm:01}
gives a necessary and sufficient condition for $u$ to be a solution
of the fractional integral equation \eqref{eq:01}, 
Theorem~\ref{thm:02} only gives a necessary condition
for $u$ to be a solution of the fractional differential 
equation \eqref{eq:06}. Indeed, as one can understand 
from the proof of Theorem~\ref{thm:02},
the bilinear form $\left(\cdot, \cdot \right)$ is degenerate
and functional $\mathcal{J}_{D}^{\alpha}[u]$ may admit
more critical points than the solutions of the given
fractional differential equation.

\begin{theo}
\label{thm:02}
If function $u$ is a solution of the fractional differential equation \eqref{eq:06},
then $u$ is a critical point of the fractional variational functional
\begin{equation}
\label{eq:09}
\mathcal{J}_{D}^{\alpha}[u] :=
\frac{1}{2}\left(\mathcal{D}^{\alpha} u,u\right) - \left(f,u\right).
\end{equation}
\end{theo}

\begin{proof}
Let $\Phi(\varepsilon) := \mathcal{J}_{D}^{\alpha}[u + \varepsilon v]$. 
Using Proposition~\ref{prop:03} and \eqref{eq:07}, we obtain that
\begin{equation*}
\Phi(\varepsilon) 
= \frac{1}{2} \left<\mathcal{D}^{\alpha} u, \mathcal{D}^{\alpha} u\right>
+ \varepsilon \left<\mathcal{D}^{\alpha} u, \mathcal{D}^{\alpha} v\right>
+ \frac{1}{2} \varepsilon^2 \left<\mathcal{D}^{\alpha} v, \mathcal{D}^{\alpha} v\right>
- \left<f, \mathcal{D}^{\alpha} u\right> - \varepsilon \left<f, \mathcal{D}^{\alpha} v\right>.
\end{equation*}
By definition, function $u$ is a critical point
of \eqref{eq:09} if $\varepsilon = 0$ is a critical point of $\Phi(\varepsilon)$.
Differentiating $\Phi(\varepsilon)$ with respect to $\varepsilon$ and putting
$\varepsilon = 0$ gives, by Fermat's theorem, that
\begin{equation}
\label{eq:10}
\begin{split}
\left<\mathcal{D}^{\alpha} u, \mathcal{D}^{\alpha} v\right>
- \left<f, \mathcal{D}^{\alpha} v\right> = 0
&\Leftrightarrow
\left<\mathcal{D}^{\alpha} u - f, \mathcal{D}^{\alpha} v\right> = 0\\
&\Leftrightarrow
\left(\mathcal{D}^{\alpha} u - f,  v\right) = 0.
\end{split}
\end{equation}
If $u$ is solution of \eqref{eq:06}, 
then $\mathcal{D}^{\alpha} u - f = 0$
and \eqref{eq:10} holds, that is,
$u$ is a critical point of 
$\mathcal{J}_{D}^{\alpha}[u]$.
\end{proof}

Theorem~\ref{thm:02} can be interpreted as a fractional analogue 
of the classical fact that symmetric operators generate 
quadratic variational functionals.

\begin{rem}
The degeneracy of $(\cdot,\cdot)$ is the fundamental reason why
Theorem~\ref{thm:02} only provides an implication and not an equivalence.
In contrast with the integral case, vanishing of
$(\mathcal{D}^\alpha u-f,v)$ for all admissible variations does not imply
uniqueness of the residual $\mathcal{D}^\alpha u-f$.
\end{rem}

The degeneracy of the bilinear 
form $(\cdot,\cdot)$ deserves further discussion.
In classical variational theory, non-degeneracy is essential to obtain
an equivalence between the Euler--Lagrange equation and the original
operator equation. In the present setting, the kernel of the operator
$\mathcal{D}^\alpha$ induces a kernel in the bilinear form
\[
(v,u)=\langle v,\mathcal{D}^\alpha u\rangle,
\]
so that critical points of the functional may exist that do not
necessarily correspond to solutions of the differential equation.
For this reason, Theorem~\ref{thm:02} provides only a necessary condition.

Nevertheless, Section~\ref{sec:rem:02}
shows that the functional \eqref{eq:09} 
admits a natural least-squares interpretation:
\[
\mathcal{J}_D^\alpha[u]
=
\frac{1}{2}\|\mathcal{D}^\alpha u-f\|^2
-\frac{1}{2}\|f\|^2,
\]
which differs from the least-squares functional only by a constant.
Therefore, minimizing $\mathcal{J}_D^\alpha$ corresponds to minimizing the
residual of the fractional differential equation.
Section~\ref{sec:rem:02} establishes a direct connection between the proposed
fractional variational formulation and the classical least-squares
approach for operator equations. It also suggests possible numerical
applications, since least-squares formulations are known to provide
stable and flexible approximation schemes for differential and integral
equations.

\section{Relation with least-squares methods}
\label{sec:rem:02}

Variational formulations play a central role in modern numerical
analysis because they provide the natural framework for Galerkin,
mixed, hybrid and least-squares approximation methods: see, e.g.,
\cite{MR2490235,MR1115205,MR0350177}.

From definitions \eqref{eq:09} and \eqref{eq:07}, 
\begin{equation*}
\begin{split}
\mathcal{J}_{D}^{\alpha}[u] 
&= \frac{1}{2}\left<\mathcal{D}^{\alpha} u,\mathcal{D}^{\alpha} u\right> 
- \left<f,\mathcal{D}^{\alpha} u\right>
= \left<\frac{1}{2} \mathcal{D}^{\alpha} u - f,\mathcal{D}^{\alpha} u\right> \\
&= \int_{a}^{b} \left(\frac{1}{2} \left(\mathcal{D}^{\alpha} u\right)(t)-f(t)\right) 
\left(\mathcal{D}^{\alpha}\left(\mathcal{D}^{\alpha} u\right)\right)^*(t) \, dt\\
&= \int_{a}^{b} \left(\frac{1}{2} \left(\mathcal{D}^{\alpha} u\right)(t)-f(t)\right) 
\left(\mathcal{D}^{2\alpha} u\right)(b-t+a) \, dt.
\end{split}
\end{equation*}

Variational formulations can be handled by several numerical methods \cite{MR4676440}.
More generally, variational formulations provide the natural
starting point for weak formulations, Galerkin approximations,
finite element methods, and least-squares methods: see, e.g.,
\cite{MR2490235,MR1115205,MR0350177}.
Here we interpret Theorem~\ref{thm:02} in terms 
of the least-squares method. Indeed,
\begin{equation*}
\begin{split}
\mathcal{J}_{D}^{\alpha}[u] 
&= \frac{1}{2}\left<\mathcal{D}^{\alpha} u,\mathcal{D}^{\alpha} u\right> 
- \left<f,\mathcal{D}^{\alpha} u\right>\\
&= \frac{1}{2} \left< \mathcal{D}^{\alpha} u - f,\mathcal{D}^{\alpha} u\right> 
- \frac{1}{2} \left<f,\mathcal{D}^{\alpha} u\right>\\
&= \frac{1}{2} \left< \mathcal{D}^{\alpha} u - f,\mathcal{D}^{\alpha} u - f\right> 
+ \frac{1}{2} \left<\mathcal{D}^{\alpha} u - f, f\right>
- \frac{1}{2} \left<f,\mathcal{D}^{\alpha} u\right>\\
&= \frac{1}{2} \left< \mathcal{D}^{\alpha} u - f,\mathcal{D}^{\alpha} u - f\right> 
- \frac{1}{2} \left<f,f\right>\\
&= \frac{1}{2} \left(\left\|\mathcal{D}^{\alpha} u - f\right\|^2 - \left\|f\right\|^2\right),
\end{split}
\end{equation*}
where $\left\|v\right\| := \sqrt{\left<v,v\right>}$
is induced by the positive definite
bilinear form $\left<\cdot,\cdot\right>$
defined in \eqref{eq:01b}. The relation
with the least-squares method is clear
since the functionals $\mathcal{J}_{D}^{\alpha}[u]$
and $\frac{1}{2} \left\|\mathcal{D}^{\alpha} u - f\right\|^2$
differ from a constant term and thus 
have the same critical points.

Follows a simple illustrative example.

\begin{example} 
We illustrate the least-squares interpretation with a simple fractional
differential equation. Let $a=0$, $b=1$, and let
$\mathcal{D}^\alpha=\mathcal{D}_{0+}^{\alpha}$ be the left Riemann--Liouville fractional
derivative of order $\alpha\in(0,1)$. Consider
\begin{equation}
\label{eq:ls-example}
\mathcal{D}_{0+}^{\alpha}u(t)=1,\qquad t\in(0,1).
\end{equation}
Since
$$
\mathcal{D}_{0+}^{\alpha}t^{\beta-1}
=
\frac{\Gamma(\beta)}{\Gamma(\beta-\alpha)}
t^{\beta-\alpha-1},
\qquad \beta>\alpha,
$$
taking $\beta=\alpha+1$ gives
$$
\mathcal{D}_{0+}^{\alpha}
\left(
\frac{t^\alpha}{\Gamma(\alpha+1)}
\right)
=1.
$$
Thus, the exact solution is
\begin{equation}
\label{eq:ls-exact}
u(t)=\frac{t^\alpha}{\Gamma(\alpha+1)} .
\end{equation}
The least-squares formulation consists in minimizing the residual
$$
R(u)=\mathcal{D}_{0+}^{\alpha}u-1
$$
through the functional
\begin{equation}
\label{eq:ls-functional}
\mathcal{L}[u]
=
\frac{1}{2}|\mathcal{D}_{0+}^{\alpha}u-1|^2 .
\end{equation}
Equivalently, in our notation, this differs from
$\mathcal{J}_D^\alpha[u]$ only by the constant $\frac{1}{2}|1|^2$.
To obtain a finite-dimensional approximation, let us look for $u$ in the
one-dimensional trial space
$$
V_1=\operatorname{span}{t^\alpha}.
$$
Thus,
$$
u_c(t)=ct^\alpha .
$$
Using
$$
\mathcal{D}_{0+}^{\alpha}t^\alpha=\Gamma(\alpha+1),
$$
we obtain
$$
R(u_c)(t)=c\,\Gamma(\alpha+1)-1 .
$$
Therefore,
$$
\mathcal L[c]
=
\frac12\int_0^1
\left(c,\Gamma(\alpha+1)-1\right)^2\,dt
=
\frac{1}{2}\left(c,\Gamma(\alpha+1)-1\right)^2 .
$$
The minimizer satisfies
$$
\frac{d\mathcal L}{dc}
=
\Gamma(\alpha+1)
\left(c,\Gamma(\alpha+1)-1\right)
=0,
$$
and hence
$$
c=\frac{1}{\Gamma(\alpha+1)}.
$$
Consequently,
$$
u_c(t)=\frac{t^\alpha}{\Gamma(\alpha+1)},
$$
which coincides with the exact solution \eqref{eq:ls-exact}.

This example shows explicitly how the least-squares formulation minimizes
the residual of the fractional differential equation. In this simple case,
the chosen trial space already contains the exact solution, and therefore
the least-squares approximation recovers it exactly.
\end{example}

From a computational perspective, the variational formulations
obtained here may provide an alternative route for the numerical
treatment of fractional equations. Indeed, once a variational
principle is available, one may exploit the large body of methods
developed for variational problems, including Galerkin,
mixed finite element, hybrid finite element, and least-squares
approaches \cite{MR2490235,MR1115205}.

Besides its theoretical interest, a variational formulation is often
the first step towards the construction of numerical approximation
schemes. This is one of the reasons why variational principles play
such an important role in applied mathematics and scientific
computing \cite{MR1115205,MR0350177}.

The representation
$$
\mathcal{J}_D^\alpha[u]
=
\frac{1}{2}\|\mathcal{D}^\alpha u-f\|^2
-\frac12\|f\|^2
$$
shows that the proposed variational formulation is closely related
to residual minimization methods. Indeed, minimizing $\mathcal{J}_D^\alpha$ 
is equivalent to minimizing the norm
of the residual
\[
R(u)=\mathcal{D}^\alpha u-f.
\]
This interpretation places our approach within the broad family of
least-squares variational principles that have been successfully used
for ordinary differential equations, partial differential equations,
inverse problems, and numerical approximation: see, e.g.,
\cite{MR2490235,MR4676440}.

From this perspective, Theorem~\ref{thm:02} provides a bridge between
fractional differential equations and variational residual
minimization methods.
This connection may be useful for the development of finite element,
spectral, and collocation schemes for fractional equations.

The investigation of such numerical schemes for fractional
equations is an interesting topic for future research.
	

\subsection*{Acknowledgements}

This work was supported by the 
Center for Research and Development 
in Mathematics and Applications (CIDMA)
under the Portuguese Foundation for Science and Technology 
(FCT, \url{https://ror.org/00snfqn58}) via grant UID/04106/2025 
(\url{https://doi.org/10.54499/UID/04106/2025}).



\end{document}